\documentclass[12pt,a4paper]{article}
\usepackage{amsfonts,amssymb,amsmath,graphicx}
\usepackage{theorem}

{\theorembodyfont{\rm}
  \newtheorem{defi}{Definition}[section]
  
  \newtheorem{exa}[defi]{Example}
  
}

  \newtheorem{thm}[defi]{Theorem}
  \newtheorem{cor}[defi]{Corollary}

\newcommand{\PP}{{\mathbb P}}

\newcommand{\RR}{{\mathbb R}}

\newcommand{\cH}{{\mathcal H}}

\newcommand{\cS}{{\mathcal S}}

\newcommand{\cV}{{\mathcal V}}

\newcommand{\ann}{{\mathrm{ann}}}

\newcommand{\rad}{{\mathrm{rad}}}
\newcommand{\Char}{{\mathrm{char}}\,}

\newcommand{\id}{{\mathrm{id}}}
\newcommand{\dis}
                 {{\mathrel{\scriptstyle{\triangle}}}}
\newcommand{\notdis}{{\not\!\!\dis}}

\newcommand{\eps}{{\varepsilon}}
\newcommand{\GF}{{\mathrm{GF}}}
\newcommand{\GL}{{\mathrm{GL}}}

\newcommand{\quer}[1]{{\overline{#1}}}
\let\phi=\varphi

\let\theta=\vartheta

\newcommand{\DelimArray}[4]{\left#1\begin{array}{*{#3}{c}}#4\end{array}\right#2}
\newcommand{\SDelimArray}[4]{\hbox{\scriptsize\arraycolsep=.5\arraycolsep
  $\left#1\!\!\begin{array}{*{#3}{c}}#4\end{array}\!\!\right#2$}}

\newcommand{\Mat}{\DelimArray()}
\newcommand{\SMat}{\SDelimArray()}

\newenvironment{proof}
    {\begin{trivlist} \item {\sl Proof:}} 
    {{}\hfill $\square$ \end{trivlist}} 

\newcounter{abbildung} 

\begin{document}
\title{Radical parallelism on projective lines\\ and non-linear models of affine spaces}
\author{Andrea Blunck \and Hans Havlicek}
\date{}


\maketitle

\begin{abstract}
We introduce and investigate an equivalence relation called
\emph{radical parallelism} on the projective line over a ring. It is
closely related with the Jacobson radical of the underlying ring. As
an application, we present a rather general approach to non-linear
models of affine spaces and discuss some particular examples.
\end{abstract}

\section{Introduction}\label{se:intro}

If two points of the projective line over a ring are non-distant then
they are also said to be \emph{parallel}. This terminology goes back
to the projective line over the real dual numbers, where parallel
points represent parallel spears of the Euclidean plane
\cite[2.4]{benz-73}. In general, this parallelism of points is not an
equivalence relation. In the present article we shall introduce
another concept of ``parallelism'' on the projective line over a
ring. In order to avoid ambiguity we call this the \emph{radical
parallelism}, since it reflects the \emph{Jacobson radical} of a ring
$R$ in terms of the projective line over $R$. The two kinds of
parallelism coincide exactly for local rings.
\par
The radical parallelism is defined and discussed in Section
\ref{se:parallel}. There are several results on the projective line
over a local ring which can be generalized to an arbitrary ring $R$
as follows: Consider radically parallel points instead of parallel
points and the Jacobson radical of $R$ instead of the only maximal
ideal of a local ring. For example, the radical parallelism is always
an equivalence relation. It is the equality relation if, and only if,
$\rad\,R=\{0\}$. Next, in Section \ref{se:cremona}, we consider a
$K$-algebra $R$ and the associated affine chain geometry. Its
automorphism group contains bijective transformations $R\to R$
(without ``exceptional points'') which are not affine
transformations, provided that $\rad\,R\neq\{0\}$ and $K\neq\GF(2)$.
In particular, when $\dim_K R$ is finite, then these mappings are
birational, i.e., they are Cremona transformations.
\par
We may regard $R$ as an affine space over $K$ and fix one of the
non-affine transformations from the above. Then $R$ together with the
images of the lines under this transformation yields a non-linear
model of the affine space $R$ over $K$. Two particular cases of such
models are investigated in detail. The first example arises from the
ring of dual numbers over $K$. It yields, for $K=\RR$, the well-known
\emph{parabola model} of the real affine plane; cf.\
\cite[p.~67]{polster+s-01}. For an arbitrary ground field $K$, a
similar parabola model is described. However, some properties of the
parabola model of the real affine plane do not hold any more if $K$
has characteristic $2$. This is due to the fact that in this case all
tangent lines of a parabola are parallel. The second example is based
upon the three-dimensional $K$-algebra of upper $2\times 2$-matrices
over $K$. As before, we obtain a kind of ``parabola model'' which can
be easily described in terms of the associated chain geometry.
\par
Throughout this paper we shall only consider associative rings with a
unit element $1$, which is inherited by subrings and acts unitally on
modules. The trivial case $1=0$ is excluded. The group of invertible
elements of a ring $R$ will be denoted by $R^*$.
\par
Let us recall the definition of the projective line over a ring $R$:
Consider the free left $R$-module $R^2$. Its automorphism group is
the group $\GL_2(R)$ of invertible $2\times 2$-matrices with entries
in $R$. A pair $(a,b)\in R^2$ is called {\em admissible}, if there
exists a matrix in $\GL_2(R)$ with $(a,b)$ being its first row.
Following \cite[p.~785]{herz-95}, the {\em projective line over\/}
$R$ is the orbit of the free cyclic submodule $R(1,0)$ under the
action of $\GL_2(R)$. So
\begin{equation*}
  \PP(R):=R(1,0)^{\GL_2(R)}
\end{equation*}
or, in other words, $\PP(R)$ is the set of all $p\subset R^2$ such
that $p=R(a,b)$ for an admissible pair $(a,b)\in R^2$. Two such pairs
represent the same point exactly if they are left-proportional by a
unit in $R$. We adopt the convention that points are represented by
admissible pairs only. (Cf.\ \cite[Proposition 2.1]{blu+h-00b} for
the possibility to represent points also by non-admissible pairs.)

The point set $\PP(R)$ is endowed with the symmetric and
anti-reflexive relation {\em distant\/} ($\dis$) defined by
\begin{equation*}
  \dis:=(R(1,0), R(0,1))^{\GL_2(R)}.
\end{equation*}
Letting $p=R(a,b)$ and $q= R(c,d)$ gives then
\begin{equation*}
  p\,\dis\, q\,\Leftrightarrow\, \Mat2{a&b\\c&d}\in \GL_2(R).
\end{equation*}
\par
The vertices of the {\em distant graph\/} on $\PP(R)$ are the points
of $\PP(R)$, two vertices of this graph are joined by an edge if, and
only if, they are distant. Given a point $p\in\PP(R)$ let
\begin{equation*}
  \PP(R)_p:=\{x\in\PP(R)\mid x\,\dis\, p\}
\end{equation*}
be the neighbourhood of $p$ in the distant graph.
\par
We shall no longer use the term ``parallel points'' in the present
paper, but we speak instead of ``non-distant points'' ($\notdis$).
The sign $\parallel$ will be used for the radical parallelism which
is defined below.
 \par
The \emph{Jacobson radical} of a ring $R$, denoted by $\rad\,R$, is
the intersection of all the maximal left (or right) ideals of $R$. It
is a two sided ideal of $R$ and its elements can be characterized as
follows:
\begin{equation}\label{eq:radikal}
 b\in\rad\,R \,\Leftrightarrow\, 1-ab\in R^*\mbox{ for all }a\in R
 \,\Leftrightarrow\, 1-ba\in R^*\mbox{ for all }a\in R;
\end{equation}
see \cite[pp.~53--54]{lam-91}.

\section{Radical parallelism}\label{se:parallel}

A point $p\in\PP(R)$ is called \emph{radically parallel} to a point
$q\in\PP(R)$ if
\begin{equation}\label{eq:def-parallel}
  x\,\dis\, p \,\Rightarrow\, x\,\dis\, q
\end{equation}
holds for all $x\in\PP(R)$. In this case we write $p \parallel q$.
Clearly, the relation $\parallel$ is reflexive and transitive; we
shall see in due course that $\parallel$ is in fact an equivalence
relation.
\par
Each matrix $\gamma\in\GL_2(R)$ determines an automorphism
$\PP(R)\to\PP(R):p\mapsto p^\gamma$ of the distant graph. Hence, by
definition,
\begin{equation}\label{eq:par-invariant}
  p\parallel q \,\Leftrightarrow\, p^\gamma \parallel q^\gamma
\end{equation}
holds for all $p,q\in\PP(R)$ and all $\gamma\in\GL_2(R)$.

The connection between the radical parallelism on $\PP(R)$ and the
Jacobson radical of $R$ is as follows:
\begin{thm}\label{thm:1}
The point $R(1,0)$ is radically parallel to $q\in\PP(R)$ exactly if
there is an element $b$ in the Jacobson radical $\rad\,R$ such that
$q=R(1,b)$.
\end{thm}
\begin{proof}
We start with a characterization of $\rad\,R$ in terms of matrices.
For all $a,b\in R$ we have
\begin{equation}\label{eq:matrix-gleichung}
  \Mat2{1    & b \\ 0 & 1}
  \Mat2{1-ba & 0 \\ a & 1}
  =
  \Mat2{1    & b \\ a & 1}.
\end{equation}
So, by (\ref{eq:radikal}), we get
\begin{equation}\label{eq:matrix-rad}
  b\in \rad\,R \,\Leftrightarrow\,
  \Mat2{1    & b \\ a & 1}
  \in\GL_2(R) \mbox{ for all }a\in R.
\end{equation}
\par
Clearly, we have
\begin{equation}\label{eq:distant-(1,0)}
  \PP(R)_{R(1,0)}=\{x\in\PP(R)\mid x \,\dis\, R(1,0)\}=\{R(a,1)\mid a\in R\}.
\end{equation}
So (\ref{eq:matrix-rad}) shows immediately that $R(1,0)$ is radically
parallel to every point $q=R(1,b)$ with $b\in\rad\,R$.
 \par
Conversely, suppose that $R(1,0)\parallel q$. Then $R(0,1)\,\dis\,
R(1,0)$ implies $R(0,1)\,\dis\, q$. So we may set $q=R(1,b)$ with
$b\in R$. Now (\ref{eq:distant-(1,0)}) and $R(1,0)\parallel q$ imply
that the right hand side of (\ref{eq:matrix-rad}) is fulfilled,
whence $b\in\rad\,R$.
\end{proof}

In order to obtain an alternative description of the radical
parallelism we consider the factor ring $R/\rad\,R=:\quer R$ and the
canonical epimorphism $R\to \quer R:a\mapsto a+\rad\,R=:\quer a$. It
has the crucial property
\begin{equation}\label{eq:einheit-invariant}
  a\in R^* \,\Leftrightarrow\, \quer a\in \quer R\,^*
\end{equation}
for all $a\in R$; cf. \cite[Proposition 4.8]{lam-91}.
\par
Now we turn to the corresponding projective lines. The mapping
\begin{equation}\label{eq:def-phiquer}
  \PP(R)\to\PP(\quer R): p=
  R(a,b)\mapsto \quer R(\quer a,\quer b)=:\quer p
\end{equation}
is well defined and surjective \cite[Proposition 3.5]{blu+h-00b}.
Furthermore, as a geometric counterpart of
(\ref{eq:einheit-invariant}) we have
\begin{equation}\label{eq:dist-invariant}
  p\,\dis\, q\,\Leftrightarrow\, \quer p\,\dis\, \quer q
\end{equation}
for all $p,q\in\PP(R)$, where we use the same symbol to denote the
distant relations on $\PP(R)$ and on $\PP(\quer R)$, respectively.
See \cite[Propositions 3.1, 3.2]{blu+h-00b}.

\begin{thm}\label{thm:2}
The mapping given by (\ref{eq:def-phiquer}) has the property
\begin{equation}\label{eq:par=quergleich}
  p \parallel q \,\Leftrightarrow\, \quer p=\quer q
\end{equation}
for all $p,q\in\PP(R)$.
\end{thm}

\begin{proof}
Let $q=R(c,d)\in\PP(R)$. Then (\ref{eq:einheit-invariant}) shows that
$\quer R(\quer 1,\quer 0)=\quer q$ if, and only if, $c\in R^*$ and
$d\in\rad\,R$ or, equivalently, $q=R(1,b)$ with
$b:=c^{-1}d\in\rad\,R$. So from Theorem \ref{thm:1} we get
\begin{equation}\label{eq:sternchen}
R(1,0)\parallel q\,\Leftrightarrow\, \quer R(\quer 1,\quer 0)=\quer q
\end{equation}
for all $q\in\PP(R)$. Now consider arbitrary points $p,q\in\PP(R)$.
There is a matrix $\gamma\in\GL_2(R)$ with $p^\gamma=R(1,0)$. We have
\begin{equation}\label{eq:sternderl}
      \quer {r^\gamma}={\quer r}\,^{\quer \gamma}
\end{equation}
for all $r\in\PP(R)$, where $\quer \gamma\in\GL_2(\overline R)$ is
obtained by applying the canonical epimorphism to the entries of the
matrix $\gamma\in\GL_2(R)$; cf.\ \cite[Proposition 3.1]{blu+h-00b}.
With (\ref{eq:par-invariant}), (\ref{eq:sternchen}), and
(\ref{eq:sternderl}) we conclude
\begin{equation}\label{eq:namenlos}
  p\parallel q
  \,\Leftrightarrow\,
  p^\gamma=R(1,0)\parallel q^\gamma
  \,\Leftrightarrow\,
  \quer {p^\gamma}= \quer R(\quer 1,\quer 0)=\quer {q^\gamma}
  \,\Leftrightarrow\,
  \quer p^{\quer \gamma}=\quer q^{\quer \gamma}
  \,\Leftrightarrow\,
  \quer p=\quer q.
\end{equation}
This completes the proof.
\end{proof}

As an immediate consequence of Theorem \ref{thm:2} we obtain:

\begin{cor}\label{cor:1}
The radical parallelism $\parallel$ on the projective line over a
ring is an equivalence relation.
\end{cor}

In particular, $\parallel$ is a symmetric relation despite its
(seemingly asymmetric) definition in formula (\ref{eq:def-parallel}).
Since $p\parallel q$ means that the neighbourhood of $p$ in the
distant graph is a subset of the neighbourhood of $q$, we get, by
virtue of this symmetry:

\begin{cor}\label{cor:2}
The neighbourhood of a point $p\in\PP(R)$ in the distant graph cannot
be a proper subset of the neighbourhood of a point $q\in\PP(R)$.
\end{cor}

Furthermore, we have
\begin{equation}\label{eq:gleichm}
 \#\{x\in\PP(R)\mid x\parallel p\}=\#\rad\,R
\end{equation}
for all $p\in\PP(R)$; in fact, Theorem \ref{thm:1} implies that
(\ref{eq:gleichm}) holds for $p=R(1,0)$, whence the assertion follows
from the transitive action of $\GL_2(R)$ on $\PP(R)$ and
(\ref{eq:par-invariant}). Thus the ``size'' of $\rad\,R$ can be
recovered from the distant graph on $\PP(R)$ as the cardinality of an
(arbitrarily chosen) class of radically parallel points. In
particular, $\parallel$ is the equality relation if, and only if,
$\rad\,R=\{0\}$.
\par
Another easy consequence of (\ref{eq:dist-invariant}) and Theorem
\ref{thm:2} is
\begin{equation}\label{eq:par->nichtdist}
  p\parallel q
                 \,\Leftrightarrow\,
  \quer p=\quer q
                 \Rightarrow
  \quer p \,\notdis\, \quer q
                 \,\Leftrightarrow\,
  p \,\notdis\, q
\end{equation}
for all $p,q\in \PP(R)$. Note that here our assumption $1\neq 0$ is
essential, since it guarantees that $\dis$ is an antireflexive
relation. (The only point of the projective line over the zero-ring
is distant to itself.) In general, however, the converse of
(\ref{eq:par->nichtdist}) is not true:

\begin{thm}\label{thm:3}
The relations ``radically parallel'' and ``non-distant'' on $\PP(R)$
coincide if, and only if, $R$ is a local ring.
\end{thm}
\begin{proof}
Since $\GL_2(R)$ acts transitively on $\PP(R)$ and leaves $\dis$ and
$\parallel$ invariant, it suffices to characterize those rings where
\begin{equation}\label{eq:par=nichtdist}
 \{ x\in\PP(R) \mid R(1,0)\,\notdis\,x \}
 =
 \{ x\in\PP(R) \mid R(1,0) \parallel x \}.
\end{equation}
Furthermore, we recall the following property: If a pair $(a,b)\in
R^2$ is admissible and so the first row of an invertible matrix, then
the first column of the inverse matrix shows that $(a,b)$ is
unimodular, i.e., there are $a',b'\in R$ such that $aa'+bb'=1$.
 \par
Now let $R$ be local. Then $R\setminus R^*=\rad\,R$, since this is
the only maximal left ideal in $R$; see \cite[Theorem 19.1]{lam-91}.
The previous remark on unimodularity, applied to the local ring $R$,
shows that at least one entry of each admissible pair $(a,b)\in R^2$
is a unit. From this we get that a point $x\in\PP(R)$ satisfies $
R(1,0)\,\notdis\,x$ if, and only if, $x=R(1,b)$ with $b\in \rad\,R$.
But this is equivalent to $R(1,0)\parallel x$ by Theorem \ref{thm:1}.
 \par
Conversely, suppose that (\ref{eq:par=nichtdist}) holds. Choose any
non-unit $b\in R$. Then $R(1,b)$ is a point and $R(1,0)\,\notdis\,
R(1,b)$ implies $b\in\rad\,R$ by Theorem \ref{thm:1}. Hence
$R\setminus R^*\subset\rad\,R$ and, since $\rad\,R\subset R\setminus
R^*$ holds trivially, $R\setminus R^*=\rad\,R$ is a two-sided ideal.
But this means that $R$ is a local ring.
\end{proof}

Corollary \ref{cor:1}, Corollary \ref{cor:2}, and (\ref{eq:gleichm})
generalize well-known results on the projective line over a local
ring. Cf.\ \cite[Proposition 2.4.1]{herz-95}.

\section{Non-linear models of affine spaces}\label{se:cremona}

In this section $R$ is a ring and $K\neq R$ is a field contained in
the centre of $R$. So $R$ is an algebra over $K$ with finite or
infinite dimension. The point set of the \emph{chain geometry}
$\Sigma(K,R)$ is the projective line over $R$, the \emph{chains} are
the $K$-sublines of $\PP(R)$; cf. \cite[p.~790]{herz-95}.
\par
We fix the point $R(1,0)=:\infty$. By (\ref{eq:distant-(1,0)}), the
mapping
\begin{equation}\label{eq:def-iota}
  \iota:R\to\PP(R)_\infty: z\mapsto R(z,1)
\end{equation}
is a bijection. We consider $R$ as an affine space over $K$.
 \par
For each subset $\cS\subset\PP(R)$ let
$(\cS\cap\PP(R)_\infty)^{\iota^{-1}}$ be the \emph{affine trace} of
$\cS$. The affine traces of the chains through $\infty$ are precisely
the so-called \emph{regular lines} $Ku+v$ ($u\in R^*$, $v\in R$);
cf.\ \cite[Proposition 3.5.3]{herz-95}. By reversing the order of the
coordinates in Theorem \ref{thm:1}, we obtain
\begin{equation*}
 (\rad\,R)^\iota =  \{x\in\PP(R)\mid x\parallel R(0,1)\},
\end{equation*}
whence $\rad\,R$ is the affine trace of $\{x\in\PP(R)\mid x\parallel
R(0,1)\}$. One can easily check that the affine trace of
$\{x\in\PP(R)\mid x\,\notdis\,R(0,1)\}$ equals $R\setminus R^*$. In
general, however, $(R\setminus R^*)^\iota$ is not equal to
$\{x\in\PP(R)\mid x\,\notdis\,R(0,1)\}$. Let, for example, $R=\RR+\RR
j$ be the ring of real anormal-complex numbers, where $j^2=1$ and
$j\notin\RR$; cf. \cite[p.~44]{benz-73}. Then
$R(1-j,1+j)\,\notdis\,R(0,1)$, but $R(1-j,1+j)\notin R^\iota$,
because $1+j$ is not invertible.
 \par
Each matrix $\gamma\in\GL_2(R)$ defines an automorphism of the chain
geometry $\Sigma(K,R)$. Let us write $D_\gamma$ for the set of all
points $z\in R$ such that $z^{\iota\gamma}\in \PP(R)_\infty$. Then
the mapping
\begin{equation}\label{eq:def-gammastrich}
 \gamma': D_\gamma\to R: z\mapsto z^{\iota\gamma\iota^{-1}}
\end{equation}
is injective, but in general the domain and the image of $\gamma'$
will be proper subsets of $R$.

\begin{thm}\label{thm:trafo}
Let $\gamma=\SMat2{a& b\\c &d}$ be a $2\times 2$-matrix over $R$.
Then the following hold:
 \begin{enumerate}
 \item
 The matrix $\gamma$ is invertible and the corresponding mapping $\gamma'$, given by
 (\ref{eq:def-gammastrich}), is defined for all points of $R$ if, and
 only if,
\begin{equation}\label{eq:gammasmatrix}
 a,d\in R^*\mbox{, and }b\in\rad\,R.
\end{equation}
 \item
If (\ref{eq:gammasmatrix}) is satisfied then the corresponding
mapping $\gamma':R\to R$ is an affine transformation for $b=0$, and a
non-affine bijective transformation for $b\in\rad\,R\setminus\{0\}$
and $K\neq\GF(2)$.
\end{enumerate}
\end{thm}

\begin{proof}
(a) Let $\gamma\in\GL_2(R)$ and suppose that $\gamma'$ is defined for
all points of $R$. So we obtain
\begin{equation}\label{eq:global}
  (\PP(R)_\infty)^\gamma\subset\PP(R)_\infty.
\end{equation}
By definition, the distant relation $\dis$ is invariant under
$\GL_2(R)$. Therefore
$(\PP(R)_\infty)^\gamma=\PP(R)_{\infty^\gamma}$, so that
(\ref{eq:global}) is equivalent to $R(a,b)=\infty^\gamma\parallel
\infty$. Thus $a\in R^*$ and $b\in\rad\,R$ by Theorem \ref{thm:2}.
Furthermore, $R(c,d)\,\dis\,R(a,b)\parallel \infty$ yields
$R(c,d)\,\dis\,\infty$, so that $d\in R^*$.
 \par
Conversely, we infer from (\ref{eq:gammasmatrix}) that $a-bd^{-1}c\in
R^*$. Hence
\begin{equation}\label{eq:faktor-gamma}
   \Mat2{-1& 0\\0 &d}
   \Mat2{1& -b\\0 &1} \Mat2{-a+bd^{-1}c& 0\\d^{-1}c &1}
 = \Mat2{a& b\\c &d}
\end{equation}
shows that $\gamma\in\GL_2(R)$. It follows from $d\in R^*$,
$b\in\rad\,R$, and (\ref{eq:radikal}) that $zb+d\in R^*$ for all
$z\in R$. Therefore $\gamma$ yields the mapping
\begin{equation}\label{eq:lineargebrochen}
 \gamma': R\to R: z\mapsto (zb+d)^{-1}(za+c)
\end{equation}
with domain $D_\gamma=R$.
\par
(b) We deduce from (a) that (\ref{eq:global}) is satisfied. By
Corollary \ref{cor:2}, applied to the points $\infty$ and
$\infty^\gamma$, it follows that
$(\PP(R)_\infty)^\gamma=\PP(R)_{\infty^\gamma}=\PP(R)_\infty$, whence
the injective mapping $\gamma'$ is bijective. There are two cases:
 \par
If $\infty= \infty^\gamma$ then $b=0$. This implies that
$\gamma':R\to R: z\mapsto d^{-1}(za+c)$ is an affine transformation;
see also \cite[Lemma 3.5.7]{herz-95}.
 \par
If $\infty\neq \infty^\gamma$ then $b\in\rad\,R\setminus\{0\}$. We
observe, as above, that the first and the third matrix on the left
hand side of (\ref{eq:faktor-gamma}) both yield affine
transformations. Hence we may confine our attention to the
transformation
\begin{equation}\label{eq:einfachetrafo}
      \beta': R\to R: z\mapsto (1-zb)^{-1}z
\end{equation}
arising from the second matrix in (\ref{eq:faktor-gamma}). Let now
$K\neq\GF(2)$. Then there is an element $k\in K\setminus\{0,1\}$. The
image of the line $K$ under $\beta'$ carries the points
$0^{\beta'}=0$, $1^{\beta'}=(1-b)^{-1}$, and
$k^{\beta'}=(1-kb)^{-1}k$. These points are non-collinear, since
$b\in\rad\,R\setminus\{0\}$ implies that $b\notin K$, whence $1-b$
and $1-kb$ are linearly independent over $K$. Thus $\beta'$ cannot be
an affine transformation.
\end{proof}

The following example shows that we cannot drop the assumption
$K\neq\GF(2)$ in Theorem \ref{thm:trafo} (b). Let
$R=\GF(2)+\GF(2)\eps$ be the ring of dual numbers over $\GF(2)$,
where $\eps^2=0$ and $\eps\notin\GF(2)$. The invertible matrix
 $\delta:=\SMat2{1 &\eps \\0&1+\eps}$
yields a transformation on $\PP(R)$ that interchanges $\infty$ with
$R(1,\eps)$, but fixes the remaining four points of $\PP(R)$. Hence
$\delta'=\id_R$ is an affine transformation, even though
$\infty^\delta\neq\infty$.

\par
Suppose that $\gamma'$ is a non-affine bijection according to Theorem
\ref{thm:trafo} (b). We obtain a \emph{non-linear model} of the
affine space on the $K$-vector space $R$ by applying the bijection
$\gamma'$ to the points and lines of this affine space. So we get a
``new'' space which has the same point set, but the $\beta'$-images
of the ``old'' lines will be the lines in the ``new'' sense. In view
of Theorem \ref{thm:trafo} (b), such non-linear models of affine
spaces are possible whenever the radical of $R$ is non-zero and
$K\neq\GF(2)$. It would be interesting to describe explicitly the
``new lines'' in a purely geometric way. However, this is beyond the
scope of this article. Below we just give two examples, one of it
generalizes the well-known \emph{parabola model} of the real affine
plane \cite[p.~67]{polster+s-01}.
\par
We have several distinguished subgroups of $\GL_2(R)$ which, by
Theorem \ref{thm:trafo}, fix $\PP(R)_\infty$ as a set. The
commutative group
\begin{equation}\label{eq:def-B}
    B:=\left\{\SMat2{1&-b\\0&1}\mid b\in\rad\,R\right\}
\end{equation}
acts regularly on the set of points that are radically parallel to
$\infty$; cf.\ Theorem \ref{thm:1}. For each $\beta\in B$ the induced
mapping $\beta':R\to R$ is given by (\ref{eq:einfachetrafo}). Next,
there is the commutative group
\begin{equation}\label{eq:def-T}
  T:=\left\{ \SMat2{1& 0\\c &1} \mid c\in R \right\}.
\end{equation}
Each $\tau\in T$ fixes $\infty$ and, by Theorem \ref{thm:trafo}, it
yields the translation $\tau':R\to R: z\mapsto z+c$. Every
translation of $R$ arises in this way. A transformation $\tau\in T$
need not fix every point $p\parallel\infty$. In fact, if $\tau$ is
the matrix in formula (\ref{eq:def-T}) then $p=R(1,b)$, with
$b\in\rad\,R$, remains fixed if, and only if,
\begin{equation}\label{eq:fix-T}
  bcb=0.
\end{equation}
For a subset $S\subset R$ let $\ann(S):=\{a\in R\mid aS=Sa=0\}$
denote the \emph{annihilator} of $S$ in $R$. So, for example, $c\in
\ann(\rad\,R)$ implies that (\ref{eq:fix-T}) is fulfilled for all
$b\in\rad\,R$. Finally, a straightforward calculation shows that
\begin{equation}\label{eq:def-N}
  N:=\left\{\SMat2{1+ n_1& 0\\n_2 &1}\mid n_1,n_2\in \ann(\rad\,R)\cap
   \rad\,R\right\}
\end{equation}
is a commutative subgroup of $\GL_2(R)$. (Observe that
$(1+n_1)(1-n_1)=1$.) Each $\nu\in N$ stabilizes $\infty$ and, by
Theorem \ref{thm:trafo}, it determines an affinity $\nu':R\to R:
z\mapsto z(1+n_1)+n_2$. The groups $N$ and $B$ have the property that
\begin{equation}\label{eq:vertauschbar}
    \nu\beta=\beta\nu
    \mbox{ for all } \nu\in N\mbox{ and all }\beta\in B.
\end{equation}
Every point $p\parallel\infty$ remains fixed under every
transformation $\nu\in N$. For, clearly, $\infty^\nu=\infty$ and,
since $p$ can be written as $\infty^\beta$ with $\beta\in B$, we
obtain $p^\nu=\infty^{\beta\nu}=\infty^{\nu\beta}=p$ from
(\ref{eq:vertauschbar}).

\par
We adopt the notation $B':=\{\beta'\mid \beta\in B\}$; $T'$ and $N'$
are defined similarly.
\par
 In the remainder of this section, we suppose that $m:=\dim_K R$ is
finite. Then the so-called \emph{cone of singularity} $R\setminus
R^*$ is an algebraic set; see \cite[Remark 3.5.4]{herz-95}. Also, the
affine trace of a chain is an \emph{affine normal rational curve} of
degree $\leq m$, provided that it has at least two points in common
with $\PP(R)_\infty$; see \cite[Theorem 3.6.5]{herz-95}. According to
\cite[p.~804]{herz-95}, the mappings given in
(\ref{eq:def-gammastrich}) are \emph{Cremona transformations}; cf.\
also \cite{benz-77} and \cite[p.~129]{benz+s+s-81}. In particular,
the mappings described in Theorem \ref{thm:trafo} (b) are bijective
Cremona transformations $R\to R$.
 \par
Let $s$ be the dimension of the Jacobson radical of $R$. Then
$1\notin\rad\,R$ implies $s\leq m-1$. All elements of $\rad\,R$ are
nilpotent; see \cite[Proposition 4.18]{lam-91}. Thus $y^{s+1}=0$ for
all $y\in\rad\,R$. So, for each $\beta\in B$, formula
(\ref{eq:einfachetrafo}) can be written in polynomial form as
\begin{equation}\label{eq:bruchfreietrafo}
  \beta': R\to R:
  z\mapsto \left(1+zb+\cdots+(zb)^{s}\right)z.
\end{equation}

\par
The final part of this section is devoted to the investigation of
two particular examples, where we are able to describe explicitly the
images of the lines under a fixed non-identical transformation
$\beta'\in B'$. It will be easy to show that non-regular lines go
over to non-regular lines and that the images of the regular lines
are ``certain'' parabolas. Our main objective is to make more precise
this last statement. We rule out, however, the field with two
elements from our discussion, because in an affine space over
$\GF(2)$ a parabola has only two points, and it would take rather
complicated formulations to include this case.

\begin{exa}\label{exa:1}
Let $R=K+K\eps$ be the ring of dual numbers over $K$, where
$K\neq\GF(2)$. This is a local commutative ring, and its radical is
$K\eps$. The lines parallel to $K\eps$ are called \emph{vertical}. In
formula (\ref{eq:bruchfreietrafo}) we may put $z=z_1+z_2\eps$ and
$b=t\eps$ with $z_1,z_2,t\in K$. Thus we get
\begin{equation}\label{eq:einfach-dual}
    \beta':R\to R: (z_1+z_2\eps)\mapsto z_1+(t z_1^2 +z_2)\eps.
\end{equation}
We assume that $t\neq 0$. All non-regular lines or, said differently,
all vertical lines are invariant under $\beta'$. In order to describe
the images of the regular lines, we consider the group $N'$. Its
transformations are obtained from (\ref{eq:def-N}) by substituting
$n_1=l_1\eps$ and $n_2=l_2\eps$, where $l_1,l_2\in K$, and this gives
\begin{equation}\label{eq:nu-strich-dual}
  \nu' : R\to R : z_1+z_2\eps\mapsto z_1+(z_1l_1+z_2+l_2)\eps.
\end{equation}
Then, either $l_1\neq 0$, whence $\nu'$ is a non-trivial shear with
the vertical axis $z_1=-l_2/l_1$, or $l_1=0$, whence $\nu'$ is a
\emph{vertical translation}, i.e.\ a translation along the vertical
line $K\eps$. Altogether, since $l_1$ and $l_2$ can be chosen
arbitrarily in $K$, the transformations in $N'$ are all the shears
with a vertical axis and all the vertical translations. From a
projective point of view this is the group of all elations whose
centre is the point at infinity of all the vertical lines.

If $L$ is a regular line then there is a $\nu'\in N'$ such that
$L^{\nu'}=K$; for if $L$ and $K$ are parallel then $\nu'$ can be
chosen as a vertical translation, and otherwise as a non-trivial
shear whose vertical axis contains the point $K\cap L$. Hence the
group $N'$ acts transitively and, by the commutativity of $N'$, even
regularly on the set of regular lines.
\par
It is clear that the image under $\beta'$ of the regular line $K$ is
a parabola $C$, say, with an equation $z_2=tz_1^2$. By the
transitivity of $N'$ on the set of regular lines and by
(\ref{eq:vertauschbar}), the set of $\beta'$-images of the regular
lines is the orbit of the parabola $C$ under the action of the group
$N'$, i.e.\ the set of all parabolas with an equation
\begin{equation}\label{eq:dual-bildpar}
  z_2=tz_1^2+l_1z_1+l_2 \mbox{ with } l_1,l_2\in K.
\end{equation}
In projective terms this is a net of conics mutually osculating at
the point at infinity of all vertical lines.
\par
For each translation $\tau':z\mapsto z+c$, $c\in R$, the point
$\infty^\beta=R(1,t\eps)$ is fixed under $\tau$, because $(t\eps)
c(t\eps)=0$; cf.\ (\ref{eq:fix-T}). But $C$ is the affine trace of a
chain through $\infty^\beta$; so $C^{\tau'}$ is the affine trace of a
chain through $\infty^{\beta\tau}=\infty^{\beta}$, whence, by the
above, $C^{\tau'}\in C^{N'}$. Therefore $C^{T'}\subset C^{N'}$. There
are two cases:
\par
If $\Char K=2$ then $C^{T'}\neq C^{N'}$, since in this case for every
parabola in $C^{T'}$ all its tangent lines are parallel to the line
$K$, whereas there is a non-trivial shear $\nu'\in N'$ which maps $C$
to a parabola whose mutually parallel tangent lines are not parallel
to $K$.
\par
Suppose now that $\Char K\neq 2$. Then equation
(\ref{eq:dual-bildpar}) can be written in the form
\begin{equation*}
  z_2+\frac{l_1^2}{4t}-l_2 =t\left(z_1+\frac{l_1}{2t}\right)^2.
\end{equation*}
Hence for each $\nu'\in N$ there exists a translation $\tau'\in T'$
with $C^{\nu'}=C^{\tau'}$. This is illustrated (with $\nu'\neq\tau'$)
in Figure \ref{abb3}. (We just proved an affine version of a theorem
on osculating conics; see \cite[2.5.4]{brau-76-1} or
\cite[Satz~2]{pauk-79}.) Thus $C^{T'}= C^{N'}$. So the mapping
(\ref{eq:einfach-dual}) leads us in a natural way to the
aforementioned \emph{parabola model} of the affine plane over $K$,
$\Char K\neq 2$. The point set of this model is the ring $R$; its
line set consists of all vertical lines together with all translates
of the parabola $C$.
{\unitlength1cm
      \begin{center}
      \begin{minipage}[t]{5cm}
         \begin{picture}(5,3.45)
         \put(0.0 ,0.0){\includegraphics[width=5cm]{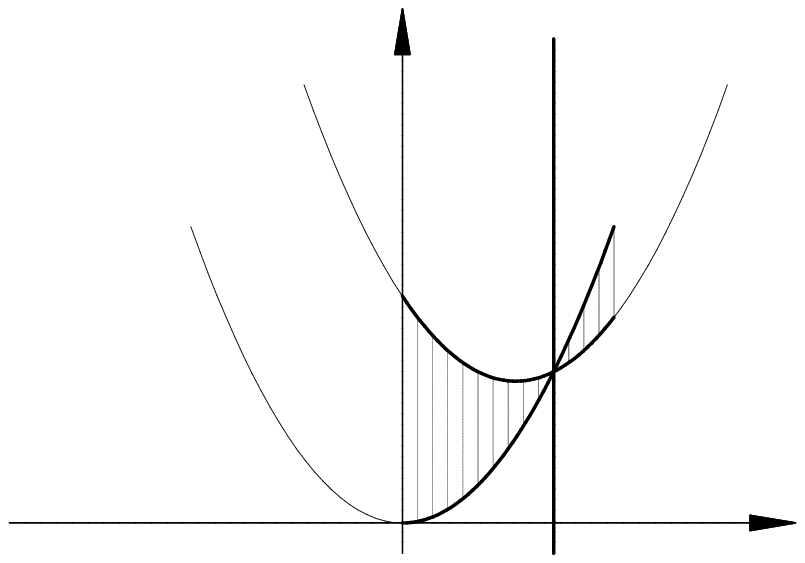}}
         \put(4.7,0.45){$z_1$}
         \put(2.65,3.4){$z_2$}
         \put(0.75,2.0){$C$}
         \put(0.3,3.0){$C^{\nu'}=C^{\tau'}$}
         \end{picture}
         {\refstepcounter{abbildung}\label{abb3}
         \centerline{Figure \ref{abb3}.}}
      \end{minipage}
      \end{center}
}%
\par
Observe that there is also a parabola model for $\Char K=2$. However,
since $C^{T'}\neq C^{N'}$, we have to use all the vertical lines and
the orbit of $C$ under the group $N'$ (rather than the translation
group) in order to obtain its line set.
\end{exa}

\par
The paper \cite{schaal-82} gives, for the real dual numbers, an
explicit description and some applications of the transformations
described in Theorem \ref{thm:trafo} (b).

\begin{exa}\label{exa:2}
Let $R$ be the ring of upper triangular $2\times 2$-matrices over
$K$, where $K\neq\GF(2)$. So, $R$ has a $K$-basis
\begin{equation*}
  j_1:=\SMat2{1&0\\ 0&0},\;
  j_2:=\SMat2{0&0 \\ 0&1},\;
  \eps:=\SMat2{0&1 \\ 0&0}.
\end{equation*}
There are two maximal ideals in $R$, namely $Kj_1+K\eps$ and
$Kj_2+K\eps$, and their union is the cone of singularity. The radical
is $K\eps$. A line or plane is said to be \emph{vertical} if it is
parallel to $K\eps$. In formula (\ref{eq:bruchfreietrafo}) we may put
$z=z_1j_1+z_2j_2+z_3\eps$ and $b=t\eps$ with $z_1,z_2,z_3,t\in K$.
Thus we get
\begin{equation}\label{eq:einfach-ternion}
    {\beta'}:R\to R: z_1j_1+z_2j_2+z_3\eps \mapsto
    z_1j_1+z_2j_2+(z_3+tz_1z_2)\eps.
\end{equation}
We assume that $t\neq 0$. All vertical lines are invariant under
$\beta'$. Each point on the cone of singularity remains fixed.
Consider a plane which is parallel to one of the planes of the cone
of singularity. The restriction of $\beta'$ to such a plane is a
planar shear, fixing the intersection of the plane with the cone of
singularity. Hence all non-regular lines go over to non-regular
lines.
 \par
The group $N'$ is obtained from (\ref{eq:def-N}) by putting
$n_1=l_1\eps$ and $n_2=l_2\eps$, where $l_1,l_2\in K$. So we get
\begin{equation}\label{}
  \nu':R\to R: z_1j_1+z_2j_2+z_3\eps\mapsto z_1j_1+z_2j_2+(z_1l_1+z_3+l_2)\eps.
\end{equation}
If $l_1\neq 0$ then $\nu'$ is a non-trivial \emph{admissible shear},
i.e.\ a shear in the direction of $K\eps$ with an axis parallel to
the plane $z_1=0$. In fact, the axis of $\nu'$ is the vertical plane
$z_1=-l_2/l_1$. If $l_1=0$ then $\nu'$ is a \emph{vertical
translation}, i.e.\ a translation along the vertical line $K\eps$.
Altogether, the transformations in $N'$ are all the admissible shears
and all the vertical translations. In projective terms this is the
group of all elations whose centre is the point at infinity of all
the vertical lines and whose axis is a plane through the line at
infinity of all the planes $z_1=\mathrm{const}$.
\par
{\unitlength1cm
      \begin{center}
      \begin{minipage}[t]{3cm}
         \begin{picture}(3,3.67)
         \put(0.0 ,0.0){\includegraphics[width=3cm]{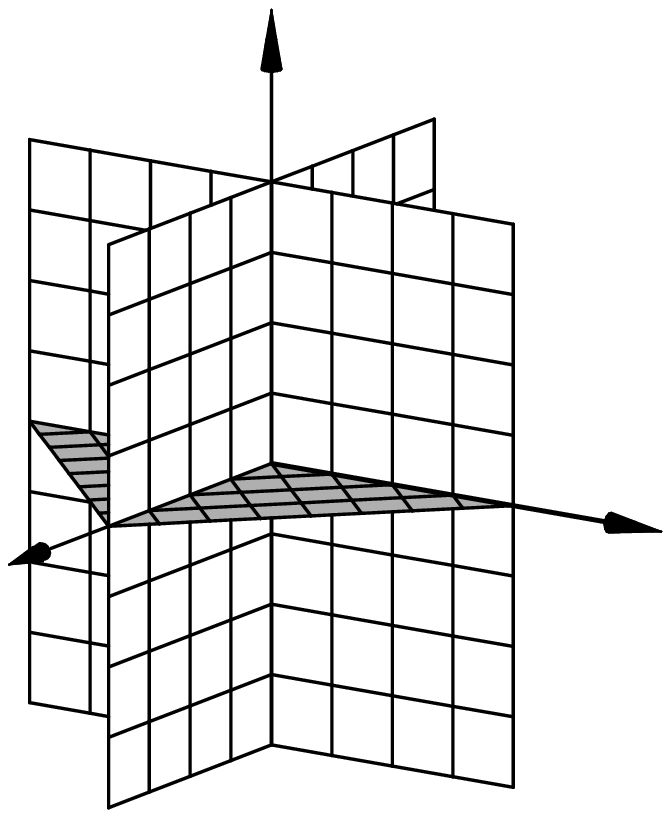}}
         \put(-0.35,1.2){$z_1$}
         \put(3.0,1.4){$z_2$}
         \put(1.35,3.5){$z_3$}
         \end{picture}
         {\refstepcounter{abbildung}\label{abb1}
         \centerline{Figure \ref{abb1}.}}
      \end{minipage}
      \hspace{2cm}
      \begin{minipage}[t]{3cm}
         \begin{picture}(3,3.67)
         \put(0.0 ,0.0){\includegraphics[width=3cm]{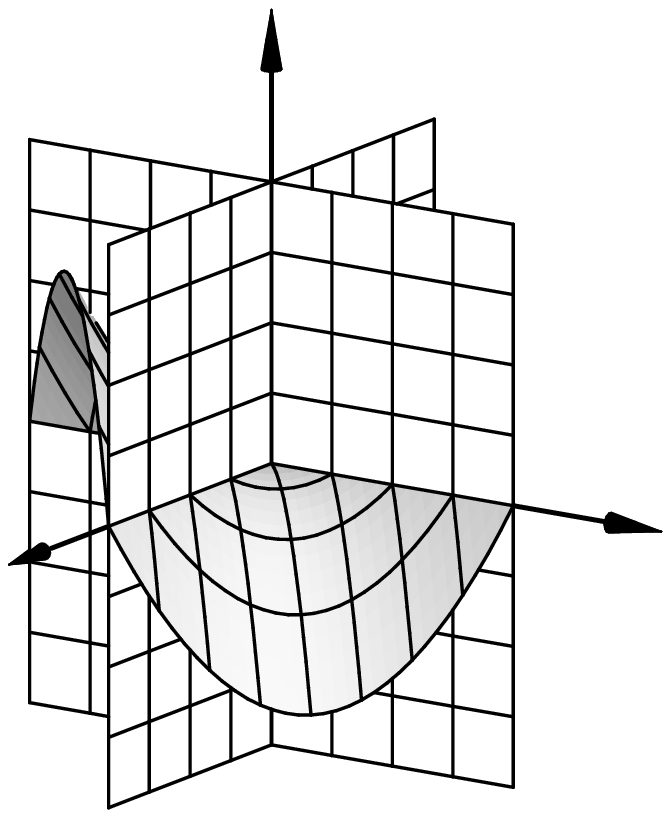}}
         \put(-0.35,1.2){$z_1$}
         \put(3.0,1.4){$z_2$}
         \put(1.35,3.5){$z_3$}
         \end{picture}
         {\refstepcounter{abbildung}\label{abb2}
          \centerline{Figure \ref{abb2}.}}
         \end{minipage}
      \end{center}
}%
Let $P$ be the plane with equation $z_3=0$. It is clear that
$P^{\beta'}$ is a hyperbolic paraboloid $H$ with equation
$z_3=tz_1z_2$, and that the set of images of the regular lines in $P$
is the set $\cH$ of all the parabolas contained in $H$. (Figure
\ref{abb1} shows, for $K=\RR$, the cone of singularity and the plane
$P$. In Figure \ref{abb2} their images under $\beta'$ are displayed.)
As in Example \ref{exa:1}, the group $N'$ operates regularly on the
set of non-vertical lines in a vertical plane which is non-parallel
to $z_1=0$: If $L$ is a regular line then there is a unique vertical
plane $V_L$ through $L$, and this plane is not parallel to $z_1=0$.
Thus there is a unique mapping $\nu\in N$ such that $L^{\nu'}=V_L\cap
P$. So, by this action of $N'$ and by (\ref{eq:vertauschbar}), the
set of $\beta'$-images of the regular lines is the union of all
orbits $C^{N'}$ with $C\in\cH$.
\par
An alternative description is possible using the translation group
$T'$. (The straightforward calculations leading to the following
results are left to the reader.) Fix a parabola $C\in\cH$ lying in
the vertical plane $V$, say. Then $\cV:=\{V^{\tau'}\cap H\mid
\tau'\in T'\}$ is a set of parabolas. It follows that each parabola
in $\cV$ is a translate of a parabola in $C^{N'}$ and vice versa.
There are two cases. If $\Char K=2$ then no parabola in
$\cV\setminus\{C\}$ is a translate of $C$. If $\Char K\neq 2$ then
all parabolas in $\cV$ are translates of $C$. (This is well known for
$K=\RR$.)
\par
Irrespective of $\Char K$, the $\beta'$-images of the regular lines
are---up to translations---precisely the parabolas in $\cH$.
Furthermore, if $\Char K\neq 2$ then this result remains true if
$\cH$ is replaced by $\cH_0:=\{C\in \cH\mid 0\in C\}$. Also we obtain
the following \emph{parabola model} of the affine $3$-space over $K$.
The point set of this model is the ring $R$; its line set consists of
all non-regular lines together with all translates of the parabolas
in $\cH$ (for arbitrary characteristic of $K$) or in $\cH_0$ (for
$\Char K\neq 2$ only).
\par
If $K=\RR$ then $R$ is isomorphic to the ring of \emph{real
ternions}. A detailed investigation of the chain geometry over the
real ternions can be found in \cite{benz-77}.
\end{exa}

The mappings discussed in Example \ref{exa:1} are closely related
with the geometry of the \emph{isotropic} (or: \emph{Galilean})
\emph{plane}. Likewise, Example \ref{exa:2} leads to a
three-dimensional Cayley-Klein geometry, namely the geometry of the
\emph{pseudo-isotropic space}. We refer, among others, to
\cite{sachs-87}, \cite{schroeder-95}, \cite[p.~136]{giering-82}, and
\cite[p.~24]{sachs-90}.
\par
The parabola model of the real affine plane is the starting point of
the theory of \emph{shift planes}. See \cite[p.~420]{salz+al-96}.
Such a plane arises, for example, from the real affine plane if the
vertical lines and the translates of a curve which is in a certain
sense ``close to a parabola'' are defined to be the ``new lines''.
Similarly, it seems plausible that ``in the neighbourhood'' of our
parabola model of the real affine $3$-space there could exist so
called \emph{$\RR^3$-spaces} (in the sense of \cite{betten-81}) other
than the real affine $3$-space. The reader should consult
\cite{klein-00} for results and a lot of references on this
interesting class of topological geometries.


\begin{thebibliography}{10}
\itemsep0pt
\bibitem{benz-73}
BENZ,~W.:
\newblock {Vorlesungen \"uber Geometrie der Algebren},
\newblock Springer, Berlin, 1973.

\bibitem{benz-77}
BENZ,~W.:
\newblock {\"U}ber eine {C}remonasche {R}aumgeometrie,
\newblock {\em Math.\ Nachr.} \textbf{80} (1977), 225--243.

\bibitem{benz+s+s-81}
BENZ,~W., SAMAGA,~H.-J., and SCHAEFFER,~H.:
\newblock Cross ratios and a unifying treatment of von {S}taudt's notion of
  reeller {Z}ug,
\newblock In Plaumann, P. and Strambach, K., editors, {Geometry -- von
  Staudt's Point of View}, pages 127--150, Reidel, Dordrecht, 1981.

\bibitem{betten-81}
BETTEN,~D.:
\newblock {T}opologische {G}eometrien auf $3$-{M}annigfaltigkeiten,
\newblock {\em Simon Stevin} \textbf{55} (1981), 221--235.

\bibitem{blu+h-00b}
BLUNCK, A. and HAVLICEK, H.:
\newblock Projective representations {I}.\ {P}rojective lines over
rings,
\newblock {\em Abh.\ Math.\ Sem.\ Univ.\ Hamburg} \textbf{70} (2000), 287--299.

\bibitem{brau-76-1}
BRAUNER,~H.:
\newblock {Geometrie projektiver R\"aume I},
\newblock BI--Wissenschaftsverlag, Mannheim, 1976.

\bibitem{giering-82}
GIERING,~O.:
\newblock {Vorlesungen \"uber h\"ohere {G}eometrie},
\newblock Vieweg, Braunschweig Wiesbaden, 1982.

\bibitem{herz-95}
HERZER,~A.:
\newblock Chain geometries,
\newblock In Buekenhout, F., editor, {Handbook of Incidence Geometry},
  pages 781--842, Elsevier, Amsterdam, 1995.

\bibitem{klein-00}
KLEIN,~H.:
\newblock Collineations of {$V$}-spaces,
\newblock {\em Geom.\ Dedicata} \textbf{83} (2000), 313--318.

\bibitem{lam-91}
LAM,~T.Y.:
\newblock {A First Course in Noncommutative Rings},
\newblock Springer, New York, 1991.

\bibitem{pauk-79}
PAUKOWITSCH,~H.P.:
\newblock {\"U}ber oskulierende {Q}uadriken und oskulierende quadratische
  {K}egel im reellen $m$-dimensionalen projektiven {R}aum,
\newblock {\em Sb.\ \"osterr.\ Akad.\ Wiss, Abt.\ II} \textbf{188} (1979), 429--450.

\bibitem{polster+s-01}
POLSTER,~B. and STEINKE,~G.:
\newblock {Geometry on Surfaces},
\newblock Cambridge University Press, Cambridge, 2001.

\bibitem{sachs-87}
SACHS,~H.:
\newblock {Ebene isotrope {G}eometrie},
\newblock Vieweg, Braunschweig Wiesbaden, 1987.

\bibitem{sachs-90}
SACHS,~H.:
\newblock {Isotrope {G}eometrie des {R}aumes},
\newblock Vieweg, Braunschweig Wiesbaden, 1990.

\bibitem{salz+al-96}
SALZMANN,~H., BETTEN,~D., GRUNDH{\"O}FER,~T., H{\"A}HL, H.,
L{\"O}WEN, R., and STROPPEL,~ M.:
\newblock {Compact Projective Planes},
\newblock de Gruyter, Berlin, 1996.

\bibitem{schaal-82}
SCHAAL,~H.:
\newblock {\"U}ber die auf der affinen {E}bene operierenden
  {L}aguerre-{A}bbildungen,
\newblock {\em Sb.\ \"osterr.\ Akad.\ Wiss.\ Abt.\ II} \textbf{191} (1982), 213--231.

\bibitem{schroeder-95}
SCHR{\"O}DER, E.M.:
\newblock Metric geometry,
\newblock In Buekenhout, F., editor, {Handbook of Incidence Geometry},
  pages 945--1013, Elsevier, Amsterdam, 1995.
\end{thebibliography}

Andrea Blunck, Fachbereich Mathematik, Universit\"at Hamburg,
Bundesstra{\ss}e 55, D-20146 Hamburg, Germany\\ Email:
\texttt{andrea.blunck@math.uni-hamburg.de}
 \par
Hans Havlicek, Institut f\"ur Geometrie, Technische Universit\"at,
Wiedner Hauptstra{\ss}e 8--10, A-1040 Wien, Austria\\ Email:
\texttt{havlicek@geometrie.tuwien.ac.at}
\end{document}